\documentclass[11pt]{amsart}

\usepackage{geometry}
\usepackage{amsfonts, amssymb, amscd}
\numberwithin{equation}{section}

\usepackage[symbol]{footmisc}

\usepackage{bm}
\usepackage{verbatim}
\usepackage{amssymb}
\usepackage{mathrsfs}
\usepackage{graphicx}
\usepackage[nospace,noadjust]{cite}
\usepackage[all]{xy}
\usepackage{tikz-cd}
\usepackage{subcaption}
\usepackage{subfiles}
\usepackage[toc,page]{appendix}
\usepackage{comment}
\usepackage{enumerate}
\usepackage{enumitem}

\usepackage{graphicx}
\graphicspath{{images/}}

\usepackage{appendix}
\usepackage{hyperref}
\usepackage{romannum}
\AtBeginDocument{\pagenumbering{arabic}}

\newcommand{\bQ}{\mathbb{Q}}

\newcommand{\cY}{\mathcal{Y}}

\newcommand{\Oo}{\mathcal{O}}
\newcommand{\cW}{\mathcal{W}}

\newcommand{\Supp}{\operatorname{Supp}}
\newcommand{\mult}{\operatorname{mult}}

\newcommand{\Pp}{\mathbb{P}}
\newcommand{\Qq}{\mathbb{Q}}

\newcommand{\Rr}{\mathbb{R}}
\newcommand{\F}{\mathbb{F}}

\newcommand{\Zz}{\mathbb{Z}}

\newcommand{\I}{\mathcal{I}}

\newcommand{\bM}{\mathbf{M}}
\newcommand{\cX}{\mathcal{X}}
\newcommand{\cB}{\mathcal{B}}

\newcommand{\alct}{a\text{-}\operatorname{lct}}

\newcommand{\WDiv}{\operatorname{WDiv}}
\newcommand{\lct}{\operatorname{lct}}

 \usepackage{todonotes}

\newtheorem{thm}{Theorem}[section]
\newtheorem{conj}[thm]{Conjecture}
\newtheorem{cor}[thm]{Corollary}
\newtheorem{lem}[thm]{Lemma}

\newtheorem{prop}[thm]{Proposition}

\theoremstyle{definition}
\newtheorem{defn}[thm]{Definition}
\newtheorem{rem}[thm]{Remark}

\theoremstyle{definition}

\usepackage{todonotes}
\begin{document}

\title{On multiplicities of fibers of Fano fibrations}

\author{Guodu Chen and Chuyu Zhou}

\address{School of Mathematical Sciences, Shanghai Jiao Tong University, 800 Dongchuan RD Shanghai, Minhang District Shanghai 200240, China}
\email{chenguodu@sjtu.edu.cn}

\address{Department of Mathematics, Yonsei University, 50 Yonsei-ro, Seodaemun-gu, Seoul 03722, Republic of Korea}
\email{chuyuzhou1@gmail.com}

\begin{abstract}
In this note, we reduce various conjectures in birational geometry, including Shokurov's conjecture on singularities of the base of log Calabi-Yau  fibrations of Fano type and boundedness conjecture for rationally connected Calabi-Yau varieties, to a conjecture on multiplicities of fibers of Fano fibrations over curves.
\end{abstract}

\date{\today}

\thanks{%2010 \emph{Mathematics Subject Classification}: 14E30, 14J45, 14J17.
}

\maketitle

\pagestyle{myheadings}\markboth{\hfill G. Chen and C. Zhou \hfill}{\hfill On  multiplicities of fibers of Fano fibrations \hfill}

\tableofcontents

\section{Introduction}

We work over the complex number field $\mathbb{C}$. 

As the outcome of minimal model program, log Calabi-Yau fibrations and Mori-Fano fibrations play important roles in the classification of algebraic varieties, and it is natural to ask how the singularities behave on the total space and the base of these fibrations (e.g. \cite{Bir18}). For Fano varieties, people have recently established many powerful theory including boundedness, complement theory (e.g. \cite{Bir19, Bir21}), and K-moduli theory (e.g. \cite{Xu21}). For Calabi-Yau varieties, however, many aspects are still mysterious, e.g. boundedness. The following conjecture due to M\textsuperscript{c}kernan-Shokurov on the singularities of the base of log Calabi-Yau fibrations of Fano type is still widely open, we refer the readers to \cite{Bir16,Bir18,HJL22} for the recent process.

\begin{conj}[M\textsuperscript{c}kernan-Shokurov]\label{conj:SM}
Let $d$ be a positive integer and $\epsilon$ a positive real number. Then there exists a positive real number $\delta$ depending only on $d$ and $\epsilon$ satisfying the following. 
	
Assume that $(X,B)$ is a log pair of dimension $d$ and $\pi:X\to Z$ is a contraction between projective normal varieties such that 
\begin{enumerate}
  \item $(X,B)$ is $\epsilon$-lc,
  \item $X$ is of Fano type over $Z$, and
  \item $K_X+B\sim_{\Rr,Z}0$.	
\end{enumerate}
Then $B_Z\le 1-\delta$, where $B_Z$ is the discriminant part of the canonical bundle formula for $(X,B)$ over $Z$. %	we can choose $M_Z\ge0$ representing the moduli part such that $(Z,B_Z+M_Z)$ is $\delta$-lc. 
\end{conj}

The above conjecture is known in some special cases. For example, if one requires that the coefficients of the boundary $B$ are bounded from below, then the general fibers of the log Calabi-Yau fibrations lie in a log bounded family, and this leads to a positive answer to the conjecture (see \cite[Theorem 1.4]{Bir16}). In Birkar's solution to this case, the key idea is to show that the multiplicities of fibers are bounded. Then it is also natural to ask the following question.

\begin{conj}\label{conj:bddfiber}
Let $d$ be a positive integer and $\epsilon$ a positive real number. Then there exists a positive integer $m$ depending only on $d$ and $\epsilon$ satisfying the following. 
	
Assume that $(X,B)$ is a log pair of dimension $d$ and $\pi:X\to C$ is a contraction such that 
\begin{enumerate}
  \item $(X,B)$ is $\epsilon$-lc,
  \item $X$ is of Fano type over $C$ where $C$ is a curve, and
  \item $K_X+B\sim_{\Rr,C}0$.	
\end{enumerate}
Then $m_i\le m$ for every $i$, where $\pi^*z=\sum_i m_iF_i$, $F_i$ are the irreducible components of $\pi^*z$, and $z\in C$ is a closed point. 
\end{conj}

It is clear that Conjecture \ref{conj:bddfiber} is a direct corollary of Conjecture \ref{conj:SM}. As we will see later, Conjecture \ref{conj:SM} provides a suitable inductive frame to prove boundedness results of some special kind of Calabi-Yau varieties (e.g. Corollary \ref{cor:basesing}).

We recall that a variety $X$ is said to be \emph{rationally connected} (RC for short) if any two general points can be connected by the image of a rational curve.

\begin{conj}\label{conj:cybdd}
Let $d$ be a positive integer and $\epsilon$ a positive real number. Then the set of rationally connected projective varieties $X$ such that $\dim X=d$, $(X,B)$ is $\epsilon$-lc and $K_X+B\sim_\Rr 0$ for some boundary $B$, is bounded up to flops.
\end{conj}

In Conjecture \ref{conj:cybdd}, if we replace the condition $K_X+B\sim_\Rr 0$ with $-K_X-B$ being nef, it is also expected that $X$ is bounded (e.g. \cite[Conjecture 3.9]{MP04}).  This case could  be put in Calabi-Yau setting in the generalized sense. Then we naturally ask the following boundedness conjecture for generalized Calabi-Yau varieties.

\begin{conj}\label{conj:gcybdd}
Let $d$ be a positive integer and $\epsilon$ a positive real number. Then the set of rationally connected projective varieties $X$ such that $\dim X=d$, $(X,B+\bM)$ is an $\epsilon$-lc g-pair and $K_X+B+\bM_X\sim_\Rr 0$ for some $\Rr$-divisor $B$ and b-$\Rr$-divisor $\bM$, is bounded in codimension one.
\end{conj}

\begin{rem}
For all above conjectures, we do not put any assumption on the coefficients of $B$ like DCC condition or the coefficients are bounded from below. All these conjectures are stated for pairs with $\Rr$ coefficients, however, up to a perturbation, in the rest of this paper, we always assume that $\epsilon$ and the coefficients of $B$ are rational.
\end{rem}

The purpose of this note is to reduce both Conjecture \ref{conj:SM}, Conjecture \ref{conj:cybdd} and Conjecture \ref{conj:gcybdd} to Conjecture \ref{conj:bddfiber}. In Proposition \ref{further reduction}, we will further reduce Conjecture \ref{conj:bddfiber} to a even `simpler' setting which is stated for Fano fibrations (see Conjecture \ref{conj:fano fibration}). Thus all conjectures stated in this note will be implied by Conjecture \ref{conj:bddfiber} or Conjecture \ref{conj:fano fibration}.

\begin{thm}{\rm{(=Proposition \ref{1.2 to 1.1})}}\label{thm1}
Conjecture \ref{conj:bddfiber} in dimension $d$ implies Conjecture \ref{conj:SM} in dimension $d$.
\end{thm}

\begin{thm}{\rm{(=Proposition \ref{1.2 to 1.3})}}\label{thm2}
Conjecture \ref{conj:bddfiber} in dimension $d$ implies Conjecture \ref{conj:cybdd} in dimension $d$.
\end{thm}

\begin{thm}{\rm{(=Proposition \ref{1.2 to 1.4})}}\label{thm3}
Conjecture \ref{conj:bddfiber} in dimension $d$ implies Conjecture \ref{conj:gcybdd} in dimension $d$.
\end{thm}

In Section 4, we also confirm Conjecture \ref{conj:gcybdd} for surfaces and threefolds. Moreover, we prove that those surfaces in Conjecture \ref{conj:gcybdd} are bounded. Remark that Conjecture \ref{conj:gcybdd} in dimension $3$ is proved by Birkar, Di Cerbo, and Svaldi \cite[Theorem 1.6]{BDCS20}.

\medskip
\noindent{\bf Acknowledgements.} 
This work began when G. Chen visited C. Zhou at EPFL in August of 2022. G. Chen would like to thank their hospitality. We would like to thank Jingjun Han and Jihao Liu for valuable discussions and suggestions. G. Chen is supported by the China post-doctoral grants BX2021269 and 2021M702925. C. Zhou is supported by grant European Research Council (ERC-804334).

\section{Preliminaries}
\subsection{Divisors}
Let $\F\in\{\Qq,\Rr\}$. Let $X$ be a normal variety and $\WDiv(X)$ the free abelian group of Weil divisors on $X$. Then an $\F$-divisor is defined to be an element of $\WDiv(X)_\F:=\WDiv(X)\otimes\F$.

A \emph{b-divisor} on $X$ is an element of the projective limit
$$\textbf{WDiv}(X)=\lim_{Y\to X}\WDiv (Y),$$
where the limit is taken over all the pushforward homomorphisms $f_*:\WDiv (Y)\to\WDiv(X)$ induced by proper birational morphisms $f:Y\to X$. In other words, a b-divisor $\textbf{D}$ on $X$ is a collection of Weil divisors $\textbf{D}_Y$ on higher models $Y$ of $X$ that are compatible under pushforward. The divisors $\textbf{D}_Y$ are called the \emph{traces} of $\textbf{D}$ on the birational models $Y.$ A b-$\F$-divisor is defined to be an element of $\textbf{WDiv}(X)\otimes\F$. 

The \emph{Cartier closure} of an $\F$-Cartier $\F$-divisor $D$ on $X$ is the b-$\F$-divisor $\overline{D}$ with trace $\overline{D}_Y=f^*D$ for any proper birational morphism $f:Y\to X$. A b-$\F$-divisor $\textbf{D}$ on $X$ is \emph{b-$\F$-Cartier} if $\textbf{D}=\overline{D_{Y}}$ where $\overline{D_Y}$ is an $\F$-Cartier $\F$-divisor on a birational model $Y$ of $X$, in this situation, we say $\textbf{D}$ \emph{descends} to $Y$. We say $\textbf{D}$ is \emph{b-nef} if it descends to a model $Y$ over $X$, such that $\textbf{D}_Y$ is nef. We say $\textbf{D}$ is \emph{b-abundant} if it descends to a model $Y$ over $X$ such that $\textbf{D}_Y$ is abundant, i.e., $\kappa(Y,\textbf{D}_Y)=\kappa_\sigma(Y,\textbf{D}_Y)$, see \cite[Chapters II,V]{Nak04} for the definitions.

\subsection{Singularities of g-pairs}
We adopt the standard notation and definitions in \cite{KM98}, and will freely use them.

Here, we discuss very briefly the analogous concepts for generalized pairs, while for further details we refer the readers to \cite{BZ16}.
\begin{defn}
A \emph{generalized sub-pair} (g-sub-pair for short) $(X,B+\bM)$ consists of an $\Rr$-divisor $B$ on $X$, and a nef b-$\Rr$-divisor $\bM$ on $X$, such that $K_X+B+\bM_X$ is $\Rr$-Cartier. If $B$ is effective, then we call $(X,B+\bM)$ a \emph{generalized pair} (g-pair for short). 

\smallskip

Let $(X,B+\bM) $ be a g-(sub-)pair and $f:W \to X$ a log resolution of $(X,\Supp B)$ to which $\bM$ descends. We may write 
$$K_{W}+B_W+\bM_W=f^*(K_X+B+\bM_X)$$
for some $\Rr$-divisor $ B_W$ on $W$. Let $E$ be a prime divisor on $W$. The \emph{log discrepancy} of $E$ with respect to $(X,B+\bM)$ is defined as 
$$a(E,X,B+\bM):=1-\mult_EB_W.$$
We say $(X,B+\bM)$ is \emph{$\epsilon$-lc} (resp., \emph{lc, klt}) for some non-negative real number $\epsilon$ if $a(E,X,B+\bM)\ge\epsilon$ (resp., $\ge0,\ >0$) for any prime divisor $E$ over $X$. 
\end{defn}

\begin{defn}
Let $(X,B+\bM)$ be a g-pair which is $a$-lc for some $a\ge0$. Assume that $D$ is an effective $\Rr$-divisor. The \emph{$a$-lc threshold} of $(X,B+\bM)$ with respect to $D$ is defined as
\begin{align*}
  \alct(X,B+\bM;D):=\sup\{t\ge0\mid&(X,(B+tD)+\bM)\text{ is $a$-lc}\}.
\end{align*}
\end{defn}

\begin{defn}   
We say $\pi: X \to Z$ is a \emph{contraction} if $X$ and $Z$ are normal quasi-projective varieties, $\pi$ is a projective morphism, and $\pi_*\Oo_X = \Oo_Z$ $(\pi$ is not necessarily birational$)$.
\end{defn}

\begin{defn}[{\cite[Definition 2.3]{PS09}}] 
Let $\pi: X \rightarrow Z$ be a contraction between normal varieties, $X$ is said to be of \emph{Fano type} over $Z$ if one of the following equivalent conditions holds:
\begin{enumerate}
  \item there exists a klt pair $(X, B)$ such that $-(K_{X}+B)$ is ample over $Z$;
  \item there exists a klt pair $(X, B')$ such that $K_{X}+B' \sim_{\Rr,Z} 0$ and $-K_{X}$ is big over $Z$.
\end{enumerate}
\end{defn}
%It is well-known that if $X$ is of Fano type over $ Z$, then one can run an MMP on any $\Rr$-Cartier $\Rr$-divisor on $X$ over $Z$, cf. \cite{BCHM10}. 

\begin{lem}\label{lem:gpairtopair}
Assume that $(X,B+\bM)$ is an $\epsilon$-lc g-pair for some positive real number $\epsilon$, such that $\bM$ is b-abundant. Then there exists a boundary $\Delta$ such that $(X,\Delta)$ is $\frac{\epsilon}{2}$-lc and $\Delta\sim_\Rr B+\bM_X$.
\end{lem}
\begin{proof}
By assumption, there exist a higher model $f:X'\to X$ such that $\bM_{X'}$ is abundant. We may write 
$$K_{X'}+B'+\bM_{X'}=f^*(K_X+B+\bM_X)$$
for some $\Rr$-divisor $B'$. Since $\bM_{X'}$ is abundant, we can choose a decomposition $\bM_{X'}\sim_\Rr E'+A'$ such that
\begin{enumerate}
  \item $E'\ge0$ and $A'\ge0$ is semi-ample, and
  \item $(X,B'+E'+A')$ is $\frac{\epsilon}{2}$-sub-lc.
\end{enumerate}
Set $\Delta:=f_*(B'+E'+A')$, then $0\le\Delta\sim_\Rr B+\bM_X$. Moreover, by the negativity lemma, we see that
$$f^*(K_X+\Delta)=K_{X'}+B'+E'+A'.$$
It follows that $(X,\Delta)$ is $\frac{\epsilon}{2}$-lc.
\end{proof}

\subsection{Canonical bundle formulas}
Recall that for a sub-lc g-sub-pair $(X,B+\bM)$ and a contraction $\pi:X\to Z$ such that $(X,B+\bM)$ is a g-pair over the generic point of $Z$, $\dim Z>0$ and $K_X+B+\bM_X\sim_\Rr\pi^*L_Z$ for some $\Rr$-Cartier $\Rr$-divisor $L_Z$ on $Z$, %by \cite[Theorem 1.4]{Fil20},
there exists a g-sub-pair $(Z,B_Z+\bM_\pi)$ such that
$$K_X+B+\bM_X\sim_\Rr\pi^*(K_Z+B_Z+\bM_{\pi,Z}),$$
where $B_Z$ is defined as 
$$B_{Z}:=\sum_P (1-\lct(X/{Z}\ni \eta_P, B+\bM;\pi^*P))P.$$
We call $B_Z$ and $\bM_\pi$ the \emph{discriminant divisor} and \emph{moduli b-divisor} of \emph{the canonical bundle formula} for $(X,B+\bM)$ over $Z$ respectively, see \cite{Fil20} for more details. If $(X,B+\bM)$ is an lc (resp., klt) g-pair, then $(Z,B_Z+\bM_\pi)$ is lc (resp., klt). We may refer the readers to \cite{Bir19,Fil20,HanLiu20,JLX22} for the recent process towards the canonical bundle formulas.

The following result is well-known \cite{Amb05}.
\begin{prop}\label{prop:cbfabundant}
Notation as above. Assume that $\bM={\bf 0}$ and $(X,B)$ is klt. Then $\bM_\pi$ is b-nef and b-abundant.
\end{prop}

\subsection{Boundedness}
A collection $\mathcal{P}$ of projective varieties is said to be \emph{bounded} (resp., \emph{bounded in codimension one}) if there exists a projective morphism $h:\cX\to S$ between schemes of finite type such that each $X\in\mathcal{P}$ is isomorphic (resp., isomorphic in codimension one) to $\mathcal X_s$ for some closed point $s\in S$. %Here by taking a normalization of X and applying Noetherian induction, we may assume that each fiber $X_s$ is normal. %For a family $\mathcal{P}$ which is bounded in codimension one, we also say $\mathcal{P}$ is \emph{bounded modulo flops}.

We say that a collection of projective log pairs $\mathcal{P}$ is \emph{log bounded} (resp., \emph{log bounded in codimension one}) if there is a quasi-projective scheme $\cX$, a reduced divisor $\cB$ on $\cX$, and a projective morphism $h: \cX \rightarrow S$, where $S$ is of finite type and $\cB$ does not contain any fiber, such that for every $(X, B) \in \mathcal{P}$, there is a closed point $s \in S$ and an isomorphism (resp., isomorphism in codimension one) $f: \cX_{s} \rightarrow X$  such that the support of  $\cB_{s}:=\cB|_{\cX_{s}}$ coincides with the support of $f_{*}^{-1} B$. Moreover, if $\mathcal{P}$ is a set of (log) Calabi-Yau pairs, we also say $\mathcal{P}$ is \emph{(log) bounded modulo flops} when $\mathcal{P}$ is (log) bounded in codimension one.

\begin{lem}\label{lem:bddrcbpfindex}
Assume that $\mathcal{P}$ is a bounded family of projective RC varieties. Then there exists a positive integer $n$ depending only on $\mathcal{P}$ such that for any $X\in\mathcal{P}$ and any Cartier divisor $D$ on $X$ with $D\equiv0$, we have that $nD\sim0$.
\end{lem}
\begin{proof}
By assumption, there exists a projective morphism $h:\cX\to S$ between schemes of finite type such that each $X\in\mathcal{P}$ is isomorphic to $X_s$ for some closed point $s\in S$. Then we can stratify $S$ into a union of finitely many constructible subsets $S=\cup_\alpha S_{\alpha}$, such that over each $S_{\alpha}$, $\cX \times_{S} S_{\alpha}$ admits a resolution $\mu_{\alpha}: \cY_{\alpha} \rightarrow \cX \times_{S} S_{\alpha}$, i.e., $(\cY_{\alpha}, E_{\alpha}:=\operatorname{Ex}(\mu_{\alpha}))$ is a simple normal crossing, and each stratum is log smooth over $S_{\alpha}$. In particular, $(\cY_{s}, E_{s}) \rightarrow X_{s}$ is a resolution for each $s \in S_{\alpha}$. %We only need to show that there exists a positive integer $n$ depending only on $\mathcal{P}$ such that for any RC variety $\cY_s$ and any Cartier divisor $D_s$ on $\cY_s$ which is numerically trivial, then it holds that $nD_s\sim0$. 

Fix an arbitrary $s \in S_\alpha$, and let $U_\alpha\subseteq S_\alpha$ be a contractible analytic neighborhood of $s$. Let $f_{U_\alpha}: \cY_{U_\alpha}:=f^{-1}(U_\alpha) \rightarrow U_\alpha$ be the restriction of $f_\alpha:\cY_\alpha\to S_\alpha$. Since $f_{U_\alpha}$ is topologically locally trivial and $U_\alpha$ is contractible, the restriction map $H^2(\cY_{U_\alpha},\Zz) \rightarrow H^2(\cY_{s}, \Zz)$ is an isomorphism. Hence we see that $H^2(\cY_s,\Zz)$ is independent of the choice of $s\in S$. Since $X$ is rationally connected, ${\rm Pic}(\cY_s)\cong H^2(\cY_s,\Zz)$. Moreover, as $H^2(\cY_s,\Zz)$ is a finitely generated $\Zz$-module, we can find a positive integer $n$ which only depends on $H^2(\cY_s,\Zz)$ such that if $D_s$ is a numerically trivial Cartier divisor on $\cY_s$, then $nD_s\sim0$. We can conclude that $n$ has the required property. This finishes the proof.
\end{proof}

The following lemma is well-known.

\begin{lem}\label{lem:bircontbdd}
Let $d$ be a positive integer and $\epsilon$ a positive real number. Assume that $(X,B+\bM)$ is a projective $\epsilon$-lc log Calabi-Yau g-pair of dimension $d$ and $\phi:X\dashrightarrow X'$ is a birational contraction. Suppose that $X'$ is bounded in codimension one. Then $X$ is bounded in codimension one.
\end{lem}
\begin{proof}
Replacing $X'$, we may assume that $X'$ is bounded. Denote $B'$ to be the strict transform of $B$ on $X'$, then it is obvious that $(X',B'+\bM)$ is an $\epsilon$-lc log Calabi-Yau g-pair. For any prime divisor $E$ on $X$ which is exceptional over $X',$ $a(E, X', B'+\bM)=a(E, X, B+\bM) \leq 1$. So we can take a birational morphism $g: X'' \rightarrow X'$ extracting all prime divisors on $X$ which are exceptional over $X'$. In particular, $X''$ is $\mathbb{Q}$-factorial and isomorphic to $X$ in codimension 1 . We may write
$$K_{X''}+B''+\bM_{X''}=g^{*}(K_{X'}+B'+\bM_{X'}) .$$
Then $(X'', B''+\bM)$ is also an $\epsilon$-lc log Calabi-Yau g-pair and $X'' \rightarrow X'$ is of Fano type. Apply \cite[Theorem 2.2]{Bir18} (see also \cite[Theorem 3.1]{HJ22}) to $X'' \rightarrow X'$, such $X''$ belongs to a bounded family, and hence $X$ is bounded in codimension one.
\end{proof}

The following lemma is a small modification of \cite[Theorem 2.2]{Bir18}.

\begin{lem}\label{lem:mfsbdd}
Let $d$ be a positive integer and $\epsilon$ a positive real number. Assume that $(X,B+\bM)$ is a projective g-pair and $\pi:X\to Z$ is a contraction between $\Qq$-factorial varieties such that
\begin{itemize}
  \item $(X,B+\bM)$ is $\epsilon$-lc of dimension $d$,
  \item $\pi:X\to Z$ is a Mori fiber space, and
  \item $K_X+B+\bM_X\sim_{\Rr}0$.
\end{itemize}
Suppose that $Z$ is bounded in codimension one. Then $X$ is bounded in codimension one.
\end{lem}
\begin{proof}
By assumption, $B+\bM_X\not\equiv0$. As $Z$ is bounded in codimension one, there exists a projective normal variety $Z'$ that is bounded, such that $Z'$ is isomorphic to $Z$ in codimension one. By the same arguments as in \cite[Lemma 3.4]{DCS16}, there exists a projective normal variety $X'$ such that
\begin{itemize}
  \item $\phi: X \dashrightarrow X'$ is isomorphic in codimension one, and
  \item $X' \rightarrow Z'$ is a Mori fiber space.
\end{itemize}
Let $B'$ be the strict transform of $B$ on $X'$, it is clear that $(X',B'+\bM)$ is $\epsilon$-lc and $K_{X'}+B'+\bM_{X'}\sim_\Rr0$. Since $Z'$ is bounded, $X'$ is bounded according to \cite[Theorem 2.2]{Bir18}. Thus $X$ is bounded in codimension one. This completes the proof.
\end{proof}

The following lemma is a small modification of \cite[Theorem 1.4]{Bir18}.

\begin{lem}\label{lem:birkarbddmodflop}
Let $d$ be a positive integer, and $\epsilon, \delta$ two positive real numbers. Let $\mathcal{P}$ be a set of $\Qq$-factorial projective varieties. Consider the set $\mathcal{P}'$ consisting of $\Qq$-factorial projective pairs $(X, B)$ such that
\begin{enumerate}
  \item $(X, B)$ is $\epsilon$-lc log Calabi-Yau of dimension $d$,
  \item the non-zero coefficients of $B$ are $\geq \delta$,
  \item there is a contraction $\pi:X\to Z$ with $Z\in\mathcal{P}$, and
  \item $X \rightarrow Z$ factors as a tower of Fano fibrations such that $X_i\to X_{i+1}$ is a Mori fiber space for each $0\leq i\leq l-1$.
  %\item $K_{X}+B \sim_{\Rr,Z} 0$, and $(Z,B_Z)$ is log Calabi-Yau for some boundary $B_Z$.
\end{enumerate}
If $\mathcal{P}$ is bounded modulo flops, then $\mathcal{P}'$ is log bounded modulo flops.
\end{lem}
\begin{proof}
By assumption, there exists a tower of Fano fibrations
$$X=:X_{0} \stackrel{\phi_0}{\longrightarrow} X_{1} \stackrel{\phi_1}{\longrightarrow} X_{2} \stackrel{\phi_2}{\longrightarrow} \cdots \stackrel{}{\longrightarrow} X_{l}:=Z,$$
i.e., the morphism $\phi_i:X_{i} \rightarrow X_{i+1}$ is a Mori fiber space for any $0 \leq i<l$. As $(X,B)$ is $\epsilon$-lc log Calabi-Yau, there exists a boundary $B_i$ on $X_i$ such that $(X_i,B_i)$ is klt log Calabi-Yau, and $K_{X_{i-1}}+B_{i-1}\sim_\Rr\phi_{i-1}^*(K_{X_i}+B_i)$ for each $1\le i\le l$, see \cite{Amb05}. Note here that $B_i\neq0$ for any $1\le i\le l-1$, as $\phi_i$ are Fano fibrations. As $Z\in\mathcal{P}$ is bounded modulo flops, there exist a projective normal variety $Z'$ which is bounded and a small birational contraction $Z\dashrightarrow Z'$. Moreover, by \cite[Lemma 3.4]{DCS16}, there exist projective normal varieties $X_0',\dots,X_l':=Z'$ such that for any $0\le i\le l-1$, it holds that
\begin{itemize}
  \item $X_i \dashrightarrow X_i'$ is a small birational contraction, and
  \item $X_i' \rightarrow X_{i+1}'$ is a Mori fiber space.
\end{itemize}
\begin{center}
  \begin{tikzcd}
      X_0 \arrow[d, "" swap,dashed]\arrow[rr]  && X_1 \arrow[d, "" swap,dashed]\arrow[rr] && X_2 \arrow[d, "" swap,dashed]\arrow[rr]  && \cdots\arrow[rr] && X_l:=Z\arrow[d, "" swap,dashed] \\
      X_0' \arrow[rr]  && X_1' \arrow[rr] && X_2' \arrow[rr]  && \cdots\arrow[rr] && X_l':=Z' 
  \end{tikzcd}
\end{center}
Denote by $B_0'$ the push-down of $B$ on $X_0'$. It is clear that $(X_0',B_0')$ is also $\epsilon$-lc log Calabi-Yau and $B_0'\ge\delta$. According to \cite[Theorem 1.4]{Bir18}, $(X_0',B_0')$ is log bounded. Therefore $X$ is bounded modulo flops and thus $(X,B)$ is log bounded modulo flops. The proof is finished.
\end{proof}

\section{From multiplicities to singularities}\label{sec3}

\subsection{Proof of Theorem \ref{thm1}}
\begin{prop}\label{1.2 to 1.1}
Conjecture \ref{conj:bddfiber} in dimension d implies  Conjecture \ref{conj:SM} in dimension d.
\end{prop}
\begin{proof}
To prove Conjecture \ref{conj:SM},  we first assume that $Z$ is a curve. 
Then it is enough to show that 
$$\lct(X,B;\pi^*z)\ge\delta$$ 
for some positive real number $\delta$ which only depends on $d$ and $\epsilon$. 
Put 
$$t:=\frac{\epsilon}{2}\text{-}\lct(X,B;\pi^*z).$$ 
Take $\phi:W\to X$ a log resulotion of $(X,B+\pi^*z)$ such that we can write
$$K_W+B_W+t(\pi^*z)_W+(1-\frac{\epsilon}{2})\sum_iE_i=\phi^*(K_X+B+t\pi^*z)+\sum_j b_j F_j,$$
where $B_W$ and $(\pi^*z)_W$ are the strict transforms of $B$ and $\pi^*z$ on $W$ respectively, $E_i\neq F_j$ are the $\phi$-exceptional prime divisors, and $b_j\ge\frac{\epsilon}{2}-1$ are real numbers. We may also write
$$K_W+B_W=\phi^*(K_X+B)+\sum_i a_iE_i+\sum_j c_j F_j,$$
for some real numbers $a_i,c_j\ge \epsilon-1$. 

Pick a small enough positive real number $\mu$ such that
$$(1-\mu)(1-\frac{\epsilon}{2})-\mu\max_i{a_i}>0.$$
It follows that 
$$\Delta_W:=B_W+(1-\mu)t (\pi^*z)_W+\sum_i((1-\mu)(1-\frac{\epsilon}{2})-\mu a_i)E_i\ge0.$$ 
Note that $(X,B+(1-\mu) t \pi^*z)$ is $\frac{\epsilon}{2}$-lc, and
$$K_W+\Delta_W=\phi^*(K_X+B+(1-\mu)t\pi^*z)+\sum_j((1-\mu)b_j+\mu c_j)F_j,$$
where $(1-\mu)b_j+\mu c_j\geq\frac{\epsilon}{2}-1$. 
As $(W,\Delta_W+\sum_jF_j)$ is dlt and 
$$K_W+\Delta_W+\sum_jF_j\sim_{X,\Qq}\sum_j(1+(1-\mu)b_j+\mu c_j)F_j\geq \frac{\epsilon}{2}\sum_jF_j,$$
we could run an MMP on $K_W+\Delta_W+\sum_jF_j$ over $X$ and it terminates with a model $X'$ over $X$. Note that the MMP contracts only $\sum_jF_j$. Therefore we have 
$$K_{X'}+\Delta'=\phi'^*(K_X+B+(1-\mu)t\pi^*z)\sim_{\Qq,Z}0,$$
where $\Delta'$ is the push-forward of $\Delta_W$ on $X'$ and we denote by $\phi':X'\to X$.

Now $(X',\Delta')$ and $\pi':X'\to Z$ satisfy the assumptions in Conjecture \ref{conj:bddfiber} by replacing $\epsilon$ with $\frac{\epsilon}{2}$, thus $n_i\le m$ for some positive integer $m$ depending only on $d$ and $\epsilon$, where 
$$\pi'^*z=\sum_i n_iE_i'\ + \ \text{(other components)}$$
and $E_i'$ are the push-forwards of $E_i$ on $X'$. It is not hard to see that, for each $i$, we have
$$tn_{i}-a_{i}=1-\frac{\epsilon}{2}.$$
Therefore we can see that
$$t=\frac{1-\frac{\epsilon}{2}+a_{i}}{n_{i}}\geq\frac{\epsilon}{2m}=:\delta.$$
We conclude that $\delta$ has the required property as 
$$\lct(X,B;\pi^*z)>\frac{\epsilon}{2}\text{-}\lct(X,B;\pi^*z)\geq\delta.$$

When $\dim Z>1$, by the proof of \cite[Lemma 3.2]{Bir16}, one can always reduce it to the case when the base is of dimension one. The proof is finished.
\end{proof}

The idea of the proof of the following result follows from \cite[Proposition 3.5]{Bir16}.

\begin{prop}\label{prop:cbfsing}
Let $d$ be a positive integer and $\epsilon$ a positive real number. Assume Conjecuture \ref{conj:bddfiber} holds in dimension $d$. Suppose that $(X,B)$ and $\pi:X\to Z$ satisfy the assumptions in Conjecture \ref{conj:SM}, then the generalized pair $(Z,B_Z+\bM_{\pi})$ is $\delta$-lc for some positive real number $\delta$ depending only on $d$ and $\epsilon$. Here $\bM_\pi$ is the moduli b-divisor of the canonical bundle formula for $(X,B)$ over $Z$.
\end{prop}
\begin{proof}
By Theorem \ref{thm1}, Conjecture \ref{conj:SM} also holds in dimension $d$. Let $\delta$ be a positive real number satisfying the property of Conjecture \ref{conj:SM} with data $d$ and $\frac{\epsilon}{2}$. We will show that $\delta$ has the required property. 

Let $Z'\to Z$ be a log resolution of $(Z,\Supp B_Z)$ such that $\bM_\pi$ descends to $Z'$. It suffices to show that $B_{Z'}\le 1-\delta.$ Let $\phi: W \rightarrow X$ be a log resolution such that the induced rational map $W \rightarrow Z'$ is a morphism. Let
$$\Delta_{W}:=B_W+(1-\frac{\epsilon}{2}) G$$
where $B_W$ is the strict transform of $B$ on $W$ and $G$ is the reduced sum of exceptional divisor of $\phi$. The pair $(W, \Delta_{W})$ is $\frac{\epsilon}{2}$-lc and we can write
$$K_{W}+\Delta_{W}=\phi^{*}(K_{X}+B)+E$$
where $E \geq 0$ and $\Supp E=\Supp G$. %In particular, if we run an MMP on $K_{W}+\Delta_{W}+\bM_W$ over $X$ (or some open subset of $X$), then the MMP terminates with $X$ (respectively with that open subset).

Let $T$ be the graph of the rational map $X \dashrightarrow Z'$, that is, $T$ is the closure in $X \times Z'$ of the graph of $X_{0} \rightarrow Z'$ where $X_{0} \subseteq X$ is the domain of $X \dashrightarrow Z'$. Since $W$ maps to both $X$ and $Z'$, we get an induced morphism $W \rightarrow T$. Let $U \subseteq Z $ be a non-empty open set over which $Z' \rightarrow Z$ is an isomorphism. We can run an MMP on $K_{W}+\Delta_{W}$ with scaling of some ample divisor over $T$ and it terminates with a model $Y$ on which $K_{Y}+\Delta_{Y}$ is nef over $T$. Since $\pi^{-1} U \subseteq X_{0}$, the morphism $T \rightarrow X$ is an isomorphism over $U$. We may also assume that $U$ intersects nothing with the image of $G$ under $W\to Z$, 
then the morphism $Y \rightarrow T$ is also an isomorphism over $U$. Note that $K_{X}+B\sim_{\bQ,Z} 0,$ we see that $ K_{Y}+\Delta_{Y} \sim_{\bQ,U} 0$.

\begin{center}
  \begin{tikzcd}[column sep = 2em, row sep = 2em]
	W\arrow[rr,dashed]\arrow[d,"\phi"]&& Y\arrow[d,"" swap]\arrow[rr,dashed] &&X'\arrow[dd,"\pi'"]\\
		X\arrow[d,"\pi"]  && T \arrow[ll,""]\arrow[d, "" swap] &&{} \\
		Z && Z'\arrow[ll,""]  && \arrow[ll, "" swap]Z'' 
  \end{tikzcd}
\end{center}

Now, by \cite{BCHM10}, we could run an MMP on $K_{Y}+\Delta_{Y}$ with scaling of some ample divisor over $Z'$ such that it ends up with a model $X'$, on which $K_{X'}+\Delta_{X'}$ is semi-ample over $Z'$ as $\Delta_{Y}$ is big over $Z'$, where $\Delta_{X'}$ is the push-forward of $\Delta_Y$ on $X'$. Recall that $K_{Y}+\Delta_{Y}\sim_{\bQ, U} 0$, we see that $Y \dashrightarrow X'$ is an isomorphism over $U$. Let $\pi': X' \rightarrow Z''$ be the canonical model of $K_{X'}+\Delta_{X'}$ over $Z'$. Then the morphism $Z'' \rightarrow Z'$ is birational. Moreover, we have that $(X', \Delta_{X'})$ is $\frac{\epsilon}{2}$-lc, and $K_{X'}+\Delta_{X'} \sim_{\bQ,Z''} 0$. Let $\Delta_{Z''}$ be the discriminant part on $Z''$ of the canonical bundle formula for $(X',\Delta_{X'})$ over $Z''$. By our choice of $\delta$, we see that $\Delta_{Z''}\le 1-\delta$. Denote by $\tau: Z''\to Z$. It is enough to show that $B_{Z''}\le \Delta_{Z''}$, where $B_{Z''}$ is given by
$$K_{Z''}+B_{Z''}+\bM_{\pi,Z''}=\tau^*(K_Z+B_Z+\bM_{\pi,Z}).$$
Let $p: V \rightarrow X$ and $q: V \rightarrow X'$ be a common resolution of $X$ and $X'$, and let
$$P:=K_{X'}+\Delta_{X'}-q_{*} p^{*}(K_{X}+B).$$
As mentioned above we have
$$K_{W}+\Delta_{W}-\phi^{*}(K_{X}+B)=E \geq 0 .$$
Now $P$ is just the pushdown of $K_{W}+\Delta_{W}-\phi^{*}(K_{X}+B)$ via the rational map $W \dashrightarrow X'$. Therefore, $P \geq 0$. From $K_{X}+B \sim_{\bQ,Z} 0$ we get
$$p^{*}(K_{X}+B)=q^{*}q_{*}p^{*}(K_{X}+B), $$
and this combined with $P \geq 0$ results in
\begin{align*}
q^{*}(K_{X'}+\Delta_{X'})-p^{*}(K_{X}+B)=q^{*}(K_{X'}+\Delta_{X'})-q^{*}q_{*}p^{*}(K_{X}+B)=q^{*} P \geq 0 .
\end{align*}
It follows that $B_{Z''}\le\Delta_{Z''}$ by the definition of the canonical bundle formula. The proof is finished.
\end{proof}

\begin{cor}[cf. {\cite{Amb05,FG12}}]\label{cor:basesing}
In the same setting as in Proposition \ref{prop:cbfsing}, there exists a boundary $\Delta_Z$ on $Z$, such that $(Z,\Delta_Z)$ is $\frac{\delta}{2}$-lc.
\end{cor}
\begin{proof}
The result follows from Lemma \ref{lem:gpairtopair}, Proposition \ref{prop:cbfabundant} and Proposition \ref{prop:cbfsing} immediately.
\end{proof}

Corollary \ref{cor:basesing} essentially tells us that, under the assumption of Conjecture \ref{conj:bddfiber}, for any given fibration $(X, B)\to Z$ as in Conjecture \ref{conj:SM}, there exists a boundary $\Delta_Z$ such that $(Z, \Delta_Z)$ is $\frac{\delta}{2}$-lc, where $\delta$ depends only on $d, \epsilon$. As we do not put any assumption on the coefficients on the boundary (e.g. DCC condition), this provides a suitable inductive frame.

\begin{conj}\label{conj:fano fibration}
Let $d$ be a positive integer and $\epsilon$ a positive real number. Then there exists a positive integer $m$ depending only on $d$ and $\epsilon$ satisfying the following. 
	
Assume that $X$ is a normal variety of dimension $d$ and $\pi:X\to C$ is a contraction such that 
\begin{enumerate}
  \item $X$ is $\epsilon$-lc,
  \item $-K_X$ is ample over $C$ where $C$ is a curve, and 
  \item every fiber of $X\to C$ is irreducible. 
\end{enumerate}
Then $m_z\le m$, where $\pi^*z=m_z F_z$, $F_z$ is the irreducible component of $\pi^*z$, and $z\in C$ is a closed point. 
\end{conj}

\begin{prop}\label{further reduction}
Conjecture \ref{conj:fano fibration} in dimension $d$ implies Conjecture \ref{conj:bddfiber} in dimension $d$.
\end{prop}

\begin{proof}
Let $\pi: (X, B)\to C$ be a contraction given in Conjecture \ref{conj:bddfiber}. Possibly replacing $X$ with a small $\Qq$-factorialization, we may assume that $X$ is $\Qq$-factorial. Write $\pi^*z=\sum m_i F_i$, we aim to show that there exists a positive integer $m$ depending only on $d$ and $\epsilon$ such that $m_i\leq m$ for every $i$. We focus on $i_0$ among these $i$. 
Run an MMP on $K_{X}+B+\eta \sum_{i\ne i_0}F_{i}$ over $C$ for some $0<\eta\ll 1$, we get a good minimal model $\phi:X\dashrightarrow X'/C$ such that  $F_{i_0}':=\phi_*F_{i_0}$ is the only component in the fiber over $z$ for $\pi': X'\to C$. Denote by $B':=\phi_*B$. Further run an MMP on $K_{X'}+(1+\eta')B'$ over $C$ for some $0<\eta'\ll 1$, we get a good minimal model $X'\dashrightarrow X''/C$ such that $-K_{X''}$ is big and nef over $C$. Replace $X''$ with its ample model with respect to $-K_{X''}$ over $C$, we may assume $-K_{X''}$ is ample over $C$. Denote by $\pi'': X''\to C$ and $B''$ the push-forward of $B'$ on $X''$,  we see that
$\pi'': (X'', B'')\to C$ satisfies all the conditions in Conjecture \ref{conj:fano fibration} and $\pi''^*z=m_{i_0}F_{i_0}''$, where $F_{i_0}''$ is the push-forward of $F_{i_0}'$ on $X''$. By Conjecture \ref{conj:fano fibration}, we see that there exists a required $m$ such that $m_{i_0}\leq m$. The proof is finished.
\end{proof}

\begin{cor}
Conjecture \ref{conj:fano fibration} in dimension d implies Conjecture \ref{conj:SM} in dimension d.
\end{cor}
\begin{proof}
Combine Proposition \ref{1.2 to 1.1} with Proposition \ref{further reduction}.
\end{proof}

\subsection{Generalized version of Conjecture \ref{conj:SM}}

\begin{conj}\label{conj:gSM}
Let $d$ be a positive integer and $\epsilon$ a positive real number. Then there exists a positive real number $\delta$ depending only on $d$ and $\epsilon$ satisfying the following. 
	
Assume that $(X,B+\bM)$ is a g-pair of dimension $d$ and $\pi:X\to Z$ is a contraction between projective normal varieties such that 
\begin{enumerate}
  \item $(X,B+\bM)$ is $\epsilon$-lc,
  \item $X$ is of Fano type over $Z$, and
  \item $K_X+B+\bM_X\sim_{\Rr,Z}0$.	
\end{enumerate}
Then $B_Z\le 1-\delta$, where $B_Z$ is the discriminant part of the canonical bundle formula for $(X,B+\bM)$ over $Z$. %	we can choose $M_Z\ge0$ representing the moduli part such that $(Z,B_Z+M_Z)$ is $\delta$-lc. 
\end{conj}

\begin{conj}\label{conj:g-bddfiber}
Let $d$ be a positive integer and $\epsilon$ a positive real number. Then there exists a positive integer $m$ depending only on $d$ and $\epsilon$ satisfying the following. 
	
Assume that $(X,B+\bM)$ is g-pair of dimension $d$ and $\pi:X\to C$ is a contraction such that 
\begin{enumerate}
  \item $(X,B+\bM)$ is $\epsilon$-lc,
  \item $X$ is of Fano type over $C$ where $C$ is a curve, and
  \item $K_X+B+\bM_X\sim_{\Rr,C}0$.	
\end{enumerate}
Then $m_i\le m$ for every $i$, where $\pi^*z=\sum_i m_iF_i$, $F_i$ are the irreducible components of $\pi^*z$, and $z\in C$ is a closed point. 
\end{conj}

\begin{prop}\label{pair to g-pair}
Conjecture \ref{conj:bddfiber} in dimension $d$ implies Conjectures \ref{conj:gSM} and \ref{conj:g-bddfiber} in dimension $d$.
\end{prop}

\begin{proof}
We first show that Conjecture \ref{conj:bddfiber} implies Conjecture \ref{conj:g-bddfiber}.

Let $(X, B+\bM)$ be a g-pair and $\pi:X\to C$ a contraction as in Conjecture \ref{conj:g-bddfiber}. Possibly replacing $X$ with a $\Qq$-factorialization, one can assume that $X$ is $\Qq$-factorial. Write $\pi^*z=\sum m_iF_i$. We focus on $i_0$ among these $i$. One can run an MMP on $K_X+B+\mu \sum_{i\ne i_0}F_{i}+\bM_X$ over $C$ for some $0<\mu\ll 1$, which ends with a good minimal model $X\dashrightarrow X'$ over $C$, cf. \cite[Lemma 4.4]{BZ16}. This MMP contracts all components in $\pi^*z$ except $F_{i_0}$. Replace $X$ with $X'$ we may assume that $\pi^*z$ is irreducible.

Then we could run MMP on $K_X+B+(1+\mu)\bM_X$ over $C$ for some $0<\mu\ll 1$ and it ends with a good minimal model $X\dashrightarrow X''$ over $C$ (e.g.  \cite[Lemma 4.4]{BZ16}). In other words, $\bM_{X''}$ is semi-ample over $C$. Note that $(X, B+\bM)$ and $(X'', B''+\bM)$ are crepant to each other, thus we may assume $\bM_X$ is semi-ample over $C$ via replacing $X$ with $X''$. Pick a general element $A$ in a very ample linear system on $C$ such that 
\begin{enumerate}
\item $\bM_X+\pi^*A$ is semi-ample, and
\item the g-pair $(X, B+(\bM+\overline{\pi^*A}))$ is still $\epsilon$-lc with nef part $\bM+\overline{\pi^*A}$.
\end{enumerate}
Let $M\in |\bM_X+\pi^*A|_\bQ$ be a general element, then it is not hard to see that the fibration $(X,B+M)\to C$ satisfies all the conditions in Conjecture \ref{conj:bddfiber}. Note that $\pi^*z=m_{i_0}F_{i_0}$, by Conjecture \ref{conj:bddfiber}, there exists a positive integer $m$ depending only on $d$ and $\epsilon$ such that $m_{i_0}\leq m$.

\medskip

Next we turn to show that Conjecture \ref{conj:bddfiber} implies Conjecture \ref{conj:gSM}. It is enough to show that Conjecture \ref{conj:g-bddfiber} implies Conjecture \ref{conj:gSM}. Indeed, the idea of the proof is essentially the same as the proof of Proposition \ref{1.2 to 1.1} but one need to use the termination of MMP for g-pairs developed in \cite{BZ16}, e.g. \cite[Lemma 4.4]{BZ16}.
\end{proof}

Using very similar arguments as in Proposition \ref{prop:cbfsing} and the tools developed in \cite{BZ16}, e.g. \cite[Lemma 4.4]{BZ16}, we can show the following result. We left the proof to the reader.

\begin{prop}\label{g-delta-lc}
Let $d$ be a positive integer and $\epsilon$ a positive real number. Assume Conjecuture \ref{conj:bddfiber} holds in dimension $d$. Suppose that $(X,B+\bM)$ and $\pi:X\to Z$ satisfy the assumptions in Conjecture \ref{conj:gSM}, then the g-pair $(Z,B_Z+\bM_{\pi})$ is $\delta$-lc for some positive real number $\delta$ depending only on $d$ and $\epsilon$. Here $\bM_\pi$ is the moduli b-divisor of the canonical bundle formula for $(X,B+\bM)$ over $Z$.
\end{prop}

\begin{cor}
Conjecture \ref{conj:fano fibration} in dimension $d$ implies Conjecture \ref{conj:gSM} in dimension $d$.
\end{cor}

\begin{proof}
The proof is a combination of Propositions \ref{pair to g-pair} and \ref{further reduction}.
\end{proof}

\begin{cor}\label{reldim1}
Conjecture \ref{conj:gSM} holds in dimension 2. More generally, Conjecture \ref{conj:gSM} holds for relative dimension one fibrations, and the g-pair $(Z,B_Z+\bM_\pi)$ is $\delta$-lc.
\end{cor}

\begin{proof}
By \cite{Bir16},  Conjecture \ref{conj:bddfiber} holds in dimension 2. Thus by 
Proposition \ref{g-delta-lc}, Conjecture \ref{conj:gSM} holds in dimension 2. By the same argument as the proof of \cite[Corollary 1.7]{Bir16}, the ``More generally'' part holds true. 
\end{proof}

\section{On boundedness of RC Calabi-Yau varieties}\label{sec4}

\subsection{Conjectures on RC Calabi-Yau varieties}

\begin{conj}\label{conj:1gap}
Let $d$ be a positive integer. Then $1$ is not an accumulation point from below for the set of minimal log discrepancies of all normal projective Calabi-Yau rationally connected varieties of dimension $d$.
\end{conj}

\begin{conj}\label{conj:rccyindex}
Let $d$ be a positive integer. Then there exists a positive integer $n$ depending only on $d$ such that for any $d$-dimensional klt Calabi-Yau rationally connected projective variety $X$, $nK_X\sim0$.
\end{conj}

\begin{conj}\label{conj:bddrccy}
Let $d$ be a positive integer. Then all $d$-dimensional klt Calabi-Yau rationally connected projective varieties form a bounded family modulo flops.
\end{conj}

\begin{prop}\label{equivalence}
The above three conjectures are equivalent.
\end{prop}

\begin{proof}
The equivalence between Conjecture \ref{conj:rccyindex} and Conjecture \ref{conj:bddrccy} is obtained by \cite[Corollary 4.3]{BDCS20}. The implication 
$4.2\Rightarrow 4.1$ is clear. The rest is devoted to the implication $4.1\Rightarrow 4.3$.

First observe that  $X$ is non-canonical (e.g. \cite[Lemma 6.3]{Jiang21}), that is there exists a prime divisor $E$ over $X$ such that $a(E,X,0) < 1$. Then we can take a crepant model $(Y,B_Y)$ of $X$ that extracts only $E$ and $Y$ is $\Qq$-factorial. By Conjecture \ref{conj:1gap}, $B_Y\ge\delta$ for some positve real number $\delta$ which only depends on $d$. By global ACC (e.g. \cite{HMX14}), $(Y,B_Y)$ is $\epsilon$-lc for some $0<\epsilon<\frac{1}{2}$ depending only on $d$.

We claim that $(Y,B_Y)$ is log bounded in codimension one. In fact, by \cite[Theorem 3.2]{DCS16}, there exists a birational contraction $\psi: Y \dashrightarrow Y'$ to a $\Qq$-factorial log Calabi-Yau pair $(Y', B_{Y'}:=\psi_{*} B_Y),$ $B_{Y'}\neq 0$ and a tower of morphisms
$$Y'=:X_{0} \stackrel{}{\longrightarrow} X_{1} \stackrel{}{\longrightarrow} X_{2} \stackrel{}{\longrightarrow} \ldots \stackrel{}{\longrightarrow} X_{k}$$
such that
\begin{itemize}
  %\item there exists a boundary $B_{i} \neq 0$ on $X_{i}$ and $(X_{i}, B_{i})$ is a $\Qq$-factorial klt log Calabi-Yau pair for any $1 \leq i<k$ ,
  \item the morphism $X_{i} \rightarrow X_{i+1}$ is a Mori fibre space for any $0 \leq i<k$, and
  \item either $\operatorname{dim} X_{k}=0$, or $\operatorname{dim} X_{k}>0$ and $X_{k}$ is $\Qq$-factorial klt Calabi-Yau.
\end{itemize}
If $\dim X_k=0$, then $Y'$ is bounded according to \cite[Theorem 1.4]{Bir18}. It implies that $Y$ is bounded in codimension one by Lemma \ref{lem:bircontbdd} and thus $(Y,B_Y)$ is log bounded in codimension one as $B_Y\ge\delta$. If $\dim X_k>0$ and $X_k$ is $\Qq$-factorial klt Calabi-Yau, then by induction on dimension, $X_k$ is bounded modulo flops. Applying Lemma \ref{lem:birkarbddmodflop}, $(Y,B_Y)$ is log bounded modulo flops. Therefore the claim holds.

We conclude that $X$ is also bounded in codimension one by the next Lemma \ref{lem:descendbdd}.
\end{proof}

The following lemma is useful in order to show that if a set of log Calabi-Yau pairs is bounded in codimension one, then the set of certain birational models of them remains bounded in codimension one. The proof is well-known to experts (cf. \cite[Theorem 5.1]{CDCHJS21}). 

\begin{lem}\label{lem:descendbdd}
Let $\epsilon$ be a positive real number. Let $\mathcal{P}$ be a set of projective $\epsilon$-lc log Calabi-Yau pairs $(Y, B_Y)$ such that $\Supp B_Y$ is a prime divisor. Denote $\mathcal{P}'$ to be the set consisting of Calabi-Yau projective varieties $X$ such that there exists a birational morphism $h: Y \rightarrow X$, and $h^{*}K_{X}=K_{Y}+B_Y$, and $(Y, B_Y) \in \mathcal{P}$. Suppose that $\mathcal{P}$ is log bounded modulo flops. Then $\mathcal{P}'$ is bounded modulo flops.
\end{lem}

\begin{proof}
By assumption $\mathcal{P}$ is bounded modulo flops, there is a quasi-projective scheme $\cY$, a prime divisor $\cB$ that does not contain any fiber of the projective morphism $\cY \to S$, where $S$ is a disjoint union of finitely many normal varieties such that for every $(Y, B_Y) \in \mathcal{P}$, there is a closed point $s \in S$ and a small birational contraction $f: \cY_{s} \dashrightarrow Y$ such that $\cB_{s}:=\cB|_{\cY_s}=f_{*}^{-1}B_Y$. We may assume that the set of those points $s$ is dense in $S$. After decomposing $S$, we may also assume that $S$ is smooth.

Fix a point $s\in S$ corresponding to some $(Y, B_Y)$,
$$K_{\cY_{s}}+f_{*}^{-1} B_Y \sim_\bQ f_{*}^{-1}(K_{Y}+B_Y) \sim_\Qq 0$$
and therefore $(\cY_{s},f_{*}^{-1} B_Y)$ is an $\epsilon$-lc log Calabi-Yau pair. Now consider a log resolution $g: \cW \rightarrow \cY$ of $(\cY, \mathcal{B})$ and denote by $\cB_\cW$ the reduced sum of strict transform of $\mathcal{B}$ and all $g$-exceptional divisors on $\cW$. Consider the $\frac{\epsilon}{2}$-lc pair $(\cW,(1-\frac{\epsilon}{2}) \cB_\cW)$. 
%There exists an open dense set $U \subset S$ such that for the point $s \in U$ corresponding to $(Y, B_Y),$ 
Up to shrinking $S$, we can suppose that $g_{s'}: \cW_{s'} \rightarrow \cY_{s'}$ is a log resolution for any $s'\in S$. Write
$$K_{\cW_s}+B_{s}=g_{s}^{*}(K_{\cY_{s}}+ f_{*}^{-1} B_Y) \sim_\Qq 0$$
where $B_{s}\leq 1-\epsilon$ and $\Supp B_s\subseteq\cB_{\cW_s}:=\cB_\cW|_{\cW_s}$. We have
$$(K_{\cW}+(1-\frac{\epsilon}{2}) \cB_\cW)|_{\cW_s} =
 K_{\cW_s}+(1-\frac{\epsilon}{2}) \cB_{\cW_s}\sim_\Qq (1-\frac{\epsilon}{2}) \cB_{\cW_s}-B_{s} .$$
Note that the support of $(1-\frac{\epsilon}{2}) \cB_{\cW_s}-B_{s}$ coincides with the support of $\cB_{\cW_s}$ which is precisely the divisors on $\cW_s$ exceptional over $X$. Hence $(K_{\cW}+(1-\frac{\epsilon}{2}) \cB_\cW)|_{\cW_s}$ has a good minimal model $X$ and thus we can run a $(K_{\cW}+(1-\frac{\epsilon}{2})\cB_\cW)$-MMP with scaling of an ample divisor over $S$ to get a relative good minimal model $\cW'$ over $S$ by \cite[Theorem 1.2]{HMX18}. Note that $\cB_{\cW_s}$ is contracted and hence $\cW'_{s}$ is isomorphic to $X$ in codimension one. Moreover, up to shrinking $S$, this MMP contracts $\cB_\cW$ only, and $\cW'_{s'}$ is a good minimal model of $(\cW_{s'},\cB_\cW|_{\cW_{s'}})$ by \cite[Lemma 6.1]{HMX18}. It implies that $\cW'_{s'}$ is isomorphic in codimension one to $X'$ which is descend from $(Y',B_{Y'})$, where $(Y',B_{Y'})$ is the pair corresponding to $s'$ for any $s'\in S$. Applying Noetherian induction, the lemma holds.
\end{proof}

\subsection{Proof of Theorem \ref{thm2}}
\begin{prop}\label{1.2 to 1.3}
Conjecture \ref{conj:bddfiber} in dimension d implies Conjecture \ref{conj:cybdd} in dimension d.
\end{prop}

\begin{proof}
To prove Conjecture \ref{conj:cybdd} in dimension $d$, we assume Conjecture \ref{conj:cybdd} holds up to dimension $d-1$.

Let $(X, B)$ be a pair in Conjecture \ref{conj:cybdd}. We first assume  $B\ne 0$. Up to small modification, we may assume that $X$ is $\bQ$-factorial.  In this case, we could run an MMP on $K_X$ to get a Mori fiber space $\pi': (X', B')\to Z$, where $B'$ is the push-forward of $B$ via $X\dashrightarrow X'$. By Corollary \ref{cor:basesing},  there exists a boundary $\Delta_Z$ on $Z$ such that $(Z, \Delta_Z)$ is an RC Calabi-Yau variety with $\delta$-lc singularities, where $\delta$ depends only on $d$ and $\epsilon$. Since Conjecture \ref{conj:cybdd} holds up to dimension up to $d-1$, we see that $Z$ is bounded in codimension one. Thus $X'$ is bounded in codimension one by Lemma \ref{lem:mfsbdd}.  By Lemma \ref{lem:bircontbdd}, $X$ is also bounded in codimension one.

We now turn to the case when $B=0$. In this case, $X$ is non-canonical, and there exists a prime divisor $E$ over $X$ such that $a(E,X,0) < 1$. Then we can take a crepant model $(Y,B_Y)$ of $X$ that extracts only $E$ and $Y$ is $\Qq$-factorial. By what we have proved, we see that $Y$ is bounded in codimension one. We aim to show that $X$ is also bounded in codimension one. Now the situation is a little different with the setting in Lemma \ref{lem:descendbdd}, as $(Y, B_Y)$ is log bounded up to flops there but we only have boundedness modulo flops for $Y$ here. That is to say, we cannot control $B_Y$ in our situation. 
This difficulty is overcome by  Lemma \ref{lem:descendbdd2}.
\end{proof}

Notice that in the following lemma, we only assume that $\mathcal{P}$ is bounded modulo flops rather than log bounded modulo flops.

\begin{lem}\label{lem:descendbdd2}
Let $\epsilon$ be a positive real number. Let $\mathcal{P}$ be a set of projective $\epsilon$-lc RC log Calabi-Yau pairs $(Y, B_Y)$ such that $\Supp B_Y$ is a prime divisor. Denote $\mathcal{P}'$ to be the set consisting of Calabi-Yau projective varieties $X$ such that there exists a birational morphism $h: Y \rightarrow X$, and $h^{*}K_{X}=K_{Y}+B_Y$, and $(Y, B_Y) \in \mathcal{P}$. Suppose that $\mathcal{P}$ is bounded modulo flops. Then $\mathcal{P}'$ is also bounded modulo flops.
\end{lem}
\begin{proof}
By \cite[Proposition 9.2]{Bir23}, there exists a positive real number $\alpha$ such that for any $(Y, B_Y)\in \mathcal{P}$, the coefficient of $B_Y$ is not less than $\alpha$. Reminding that $\{Y| (Y, B_Y)\in \mathcal{P}\}$ is bounded, we thus see that $\{B_Y| (Y, B_Y)\in \mathcal{P}\}$ is bounded. This implies that the set $\mathcal{P}$ is log bounded and we are done by Lemma \ref{lem:descendbdd}.
\end{proof}

\subsection{Proof of Theorem \ref{thm3}}

\begin{prop}\label{1.2 to 1.4}
Conjecture \ref{conj:bddfiber} in dimension $d$ implies Conjecture \ref{conj:gcybdd} in dimension $d$.
\end{prop}

\begin{proof}

Possibly replacing $X$ by a small $\Qq$-factorialization, we may assume that $X$ is $\bQ$-factorial.  

We first assume $B+\bM_X\not\equiv0$. In this case, we could run an MMP on $K_X$ to get a Mori fiber space $\pi': (X', B'+\bM)\to Z$, where $B'$ is the push-forward of $B$ via $X\dashrightarrow X'$. By Proposition \ref{g-delta-lc},  there is a g-pair $(Z, \Delta_Z+\bM_{\pi'})$ which is $\delta$-lc Calabi-Yau and $Z$ is RC, and $\delta$ depends only on $d$ and $\epsilon$. By induction, we see that $Z$ is bounded in codimension one, thus $X'$ is bounded in codimension one by Lemma \ref{lem:mfsbdd}.  By Lemma \ref{lem:bircontbdd}, $X$ is also bounded in codimension one.

If $B+\bM_X\equiv0$, it follows from Theorem \ref{thm2}.
\end{proof}

\subsection{Conjecture \ref{conj:gcybdd} in dimension $\le 3$}

\begin{thm}\label{thm:gcybddforsurf}
Conjecture \ref{conj:gcybdd} holds in dimension $2$. Moreover, those $X$ are bounded.
\end{thm}

\begin{proof}
By \cite[Lemma 2.4]{HanLiu20a}, $(X,B)$ is klt and thus $X$ is $\Qq$-factorial. Then we can pick a positive real number $\epsilon'$ such that $(X,B+(1+\epsilon')\bM)$ is klt. We may run an MMP on $K_X+B+(1+\epsilon')\bM_X\sim_\Rr -\epsilon'(K_X+B)$ and reach a model $X'$ on which $-(K_{X'}+B')$ is nef, where $B'$ is the strict transform of $B$ on $X'$. According to \cite{Ale94}, $X'$ is bounded. From \cite[Theorem 2.2]{Bir18}, we see that $X$ is also bounded.
\end{proof}

\begin{thm}
Conjecture \ref{conj:gcybdd} holds in dimension $3$.
\end{thm}

\begin{proof}
%It is clear that Conjecture \ref{conj:gcybdd} holds up to dimension 2. We only need to prove Conjecture \ref{conj:gcybdd} holds in dimension 3. 
If $B+\bM_X\equiv0$, then we are done by \cite[Theorem 1.6]{Jiang21}. If $B+\bM_X\not\equiv0$, then we could run an MMP on $K_X$ to get a Mori fiber space $\pi': (X', B'+\bM)\to Z$, where $B'$ is the push-forward of $B$ via $X\dashrightarrow X'$. We claim that $Z$ is bounded. Assume the claim holds, then $X$ is bounded in codimension one by Lemmas \ref{lem:bircontbdd} and \ref{lem:mfsbdd}. We only need to prove the claim. If $\dim Z=0$, there is nothing to prove. If $\dim Z=1$, then $Z$ is either a rational curve or an elliptic curve. If $\dim Z=2$, then there is a g-pair $(Z,B_Z+\bM_{\pi'})$ which is $\delta$-lc for some $\delta>0$ depending only on $d$ and $\epsilon$ by Corollary \ref{reldim1}. It follows that $Z$ is bounded by Theorem \ref{thm:gcybddforsurf}. This finishes the proof.
\end{proof}

%\begin{rem}The above theorem gives a different proof of \cite[Theorem 1.6]{BDCS20}.\end{rem}

\section{Further remark on Conjecture \ref{conj:rccyindex}}

This section is devoted to Conjecture \ref{conj:rccyindex}. Since g-pair structure appears naturally if we want to prove Conjecture \ref{conj:rccyindex} by induction on the dimension via canonical bundle formula, the following conjecture is also expected.

\begin{conj}\label{conj:indexforgpair}
Let $d, p$ be two positive integers and $\I \subseteq[0,1] \cap \Qq$ a DCC set. Then there exists a positive integer $n$ depending only on $d, p$ and $\I$ satisfying the following.

Assume that $(X, B+\bM)$ is a projective RC lc g-pair of dimension $d$ such that $B \in \I$ and $p \bM$ is b-Cartier, and $K_X+B+\bM_X\equiv0$. Then $n(K_{X}+B+\bM_{X}) \sim 0$. 
\end{conj}

It is worth mentioning that Conjecture \ref{conj:indexforgpair} fails if we remove the ``RC'' assumption. For example, one can take $X$ to be an elliptic curve, $B=0$ and $\bM_X$ a non-torsion numerically trivial Cartier divisor on $X$.

In the rest of this paper, we show that Conjecture \ref{conj:rccyindex} is equivalent to Conjecture \ref{conj:indexforgpair}. Recall that we say a g-pair $(X, B+\bM)$ is \emph{dlt} if it is lc and there is a closed subset $V\subseteq X$ ($V$ can be equal to $X$) such that
\begin{itemize}
    \item $(X \backslash V,B|_{X \backslash V})$ is log smooth, and
    \item if $a(E, X, B+\bM)=0$, then the center of $E$ satisfies $\operatorname{Center}_X(E) \not \subset V$ and $\operatorname{Center}_X(E) \backslash V$ is an lc center of $(X \backslash V, B|_{X \backslash V})$.
\end{itemize}

\begin{prop}\label{prop:reduction}
Let $d$ be a positive integer. Assume Conjecture \ref{conj:indexforgpair} in dimension $\le d-1$. Then Conjecture \ref{conj:indexforgpair} holds in dimension $d$ if either $(X,B+\bM)$ is non-klt lc or $B+\bM_X\not\equiv0$. 
%In particular, Conjecture \ref{conj:indexforgpair} holds in dimension $4$ when $(X,B+\bM)$ is non-klt lc.
\end{prop}

\begin{proof}
Possibly replacing $(X,B+\bM)$ with a dlt modification (see \cite[Proposition 3.9]{HL18}), we may assume that $(X,B+\bM)$ is  $\Qq$-factorial and dlt. Moreover, $B+\bM_{X}\not\equiv0$ by assumption. In particular, $K_X\equiv-(B+\bM_X)$ is not pseudo-effective. Therefore we can run an MMP on $K_X$ which terminates with a Mori fiber space $\pi:X'\to Z$. Since $X$ is klt, so is $X'$ and hence $X'$ is of Fano type over $Z$. Denote by $B'$ the push-forward of $B$ on $X'$.

According to the Global ACC for g-pairs (see \cite[Theorem 1.6]{BZ16}), there exists a finite set $\I_0\subseteq\I$ which only depends on $d$ and $p$ such that $B\in\I_0$. Possibly replacing $p$ by a multiple, we may assume that $p\I_0\subseteq\Zz$. 

If $Z$ is a point, then by \cite[Theorem 1.10]{Bir19}, there exists a positive integer $n_1$ which only depends on $d$ and $p$ such that $n_1(K_{X'}+B'+\bM_{X'})\sim0$. Therefore $n_1(K_X+B+\bM_X)\sim0$ as $(X,B+\bM)$ and $(X',B'+\bM)$ are crepant. In the following, we may assume that $\dim Z\ge1$.

Now $K_{X'}+B'+\bM_{X'}\sim_{\Qq,Z}0$, $pB'$ is integral and $p\bM$ is b-Cartier. By \cite[Proposition 3.3]{HJ22}, there exists a hyperstandard set $\Phi\subseteq[0,1]\cap\Qq$ and a positive integer $q$ depending only on $d$ and $p$ such that there exists an lc g-pair $(Z,B_Z+\bM_\pi)$ such that
$$q(K_{X'}+B'+\bM_{X'}) \sim q\pi^{*}(K_{Z}+B_Z+\bM_{\pi,Z})\equiv0,$$
$B_Z\in\Phi$ and $q\bM_\pi$ is b-Cartier. Recall that $X$ is rationally connected, $Z$ is also rationally connected. Now by our assumption, one can find a positive integer $n_2$ depending only on $d-1,\Phi$ and $q$ such that $n_2(K_Z+B_Z+\bM_{\pi,Z})\sim0$. Then
$$0\sim n_2q\pi^*(K_Z+B_Z+\bM_{\pi,Z})\sim n_2q(K_X+B+\bM_X).$$
Therefore $n:=n_1n_2q$ has the required property. The proof is finished.
\end{proof}

\begin{prop}\label{prop:indexred}
Conjecture \ref{conj:rccyindex} in dimension $d$ implies Conjecture \ref{conj:indexforgpair} in dimension $d$.
\end{prop}

\begin{proof}
If $d=1$, then $X$ is indeed $\Pp^1$ as $X$ is rationally connected by our assumption. In this case, Conjecture \ref{conj:indexforgpair} is clear. In what follows, we may assume that $d\ge2$.

Possibly replacing $(X,B+\bM)$ by a dlt modification, we may assume that $X$ is $\Qq$-factorial. In particular, $X$ is klt. If $B+\bM_X\not\equiv0$, then there exists a positive integer $n_1$ has the required property by Proposition \ref{prop:reduction}.

Now assume that $B=0$ and $\bM_X\equiv0$, then $K_X\sim_\Qq0$. As we assume Conjecture \ref{conj:rccyindex} in dimension $d$, we may find a positive integer $n_2$ which only depends on $d$ such that $n_2K_X\sim0$. According to \cite[Theorem 1.4]{BDCS20}, there exists a bounded family $\mathcal{P}$ depending only on $d$ and $n_2$ such that $X$ and $X'$ are isomorphic in codimenison one, for some $X'\in\mathcal{P}$. Possibly replacing $X$ with $X'$, we may assume that $X$ belongs to the bounded family $\mathcal{P}$. Since $\bM_X\equiv0$, there exists a positive integer $l_0$ depending only on $\mathcal{P}$ such that $l_0\bM_X$ is Cartier, cf. \cite[Lemma 2.25]{Bir19}. Moreover, by Lemma \ref{lem:bddrcbpfindex}, there exists a positve integer $n_3$ which only depends on $\mathcal{P}$ such that $n_3l_0\bM_X\sim0$. It follows that $n_2n_3l_0(K_X+\bM_X)\sim0$.

Therefore $n:=n_1n_2n_3l_0$ has the required property.
\end{proof}

\begin{cor}
Conjecture \ref{conj:indexforgpair} holds in the following cases.
\begin{enumerate}
  \item $d\le3$.
  \item $d=4$ and either $(X,B+\bM)$ is non-klt lc or $B+\bM_X\not\equiv0$.
\end{enumerate}
\end{cor}
\begin{proof}
By \cite{Jiang21}, Conjecture \ref{conj:rccyindex} holds in dimension $3$, thus we have (1). (2) follows from (1) and Propostion \ref{prop:indexred} immediately.
\end{proof}

As a consequence, the following result holds. Note that in the following corollary, we do not require that the DCC coefficients are rational numbers, see \cite[Theorem 1.9]{HLL22} for a more general statement for threefold pairs.

\begin{cor}
Let $d,p$ be two positive integers, and $\I\subseteq[0,1]$ a DCC set. Assume that either $d\le3$ or Conjecture \ref{conj:rccyindex} holds in dimension $d$. Then the set of projective normal rationally connected varieties $X$ such that
\begin{enumerate}
    \item $(X, B+\bM)$ is klt g-pair of dimension $d$ for some boundary $B$ and some nef part $\bM$,
    \item $B \in \I$ and $p \bM$ is b-Cartier, and
    \item $K_X+B+\bM_X\equiv0$,
\end{enumerate}
is bounded in codimension one.
\end{cor}
\begin{proof}
By \cite[Remark 5.7]{CHL22} (cf. \cite{HLS19}), there exists a finite set $\I_0\subseteq[0,1]\cap\Qq$ depending only on $d,p$ and $\I$ such that we may find a boundary $B'$ such that $(X,B'+\bM)$ is klt g-pair and $K_X+B'+\bM_X\equiv0$. Since we assume Conjecture \ref{conj:rccyindex} in dimension $d$, Conjecture \ref{conj:indexforgpair} also holds in dimensiond $d$ by Proposition \ref{prop:indexred}. Therefore one can find a positive integer $p|n$ which only on $d,p$ and $\I_0$ such that $n(K_X+B'+\bM_X)\sim0$. We conclude the statement by \cite[Theorem 1.4]{HJ22}.
\end{proof}

Finally, we remark here that the Conjecture \ref{conj:rccyindex} is expected for lc log Calabi-Yau varieties, that is, we do not need the varieties to be RC, cf. \cite{XYN19,JL21}.

\newcommand{\etalchar}[1]{$^{#1}$}


\begin{thebibliography}{CDCH{\etalchar{+}}21}

\bibitem[Ale94]{Ale94}
Valery Alexeev.
\newblock Boundedness and K$^2$ for log surfaces.
\newblock {\em Internat. J. Math.}, 5(6):779–810, 1994.

\bibitem[Amb05]{Amb05}
Florin Ambro.
\newblock The modulib-divisor of an lc-trivial fibration.
\newblock {\em Compos. Math.}, 140(2):385--403, 2005.

\bibitem[BCHM10]{BCHM10}
Caucher Birkar, Paolo Cascini, Christopher~D. Hacon, and James
  M\textsuperscript{c}Kernan.
\newblock Existence of minimal models for varieties of log general type.
\newblock {\em J. Amer. Math. Soc.}, 23(2):405--468, 2010.

\bibitem[BDCS20]{BDCS20}
Caucher Birkar, Gabriele Di~Cerbo, and Roberto Svaldi.
\newblock Boundedness of elliptic {C}alabi--{Y}au varieties with a rational
  section.
\newblock {\em arXiv:2010.09769}, 2020.

\bibitem[Bir16]{Bir16}
Caucher Birkar.
\newblock Singularities on the base of a {F}ano type fibration.
\newblock {\em J. Reine Angew. Math.}, 715:125--142, 2016.

\bibitem[Bir18]{Bir18}
Caucher Birkar.
\newblock Log {C}alabi-{Y}au fibrations.
\newblock {\em arXiv:1811.10709}, 2018.

\bibitem[Bir19]{Bir19}
Caucher Birkar.
\newblock Anti-pluricanonical systems on {F}ano varieties.
\newblock {\em Ann. of Math. (2)}, 190(2):345--463, 2019.

\bibitem[Bir21]{Bir21}
Caucher Birkar.
\newblock Singularities of linear systems and boundedness of {F}ano varieties.
\newblock {\em Ann. of Math. (2)}, 193:347--405, 2021.

\bibitem[Bir23]{Bir23}
Caucher Birkar.
\newblock Singularities on Fano fibrations and beyond.
\newblock {\em arXiv:2305.18770v2}, 2023.

\bibitem[BZ16]{BZ16}
Caucher Birkar and De-Qi Zhang.
\newblock Effectivity of {I}itaka fibrations and pluricanonical systems of
  polarized pairs.
\newblock {\em Publ. Math. Inst. Hautes \'Etudes Sci.}, 123:283--331, 2016.

\bibitem[CDCH{\etalchar{+}}21]{CDCHJS21}
Weichung Chen, Gabriele Di~Cerbo, Jingjun Han, Chen Jiang, and Roberto Svaldi.
\newblock Birational boundedness of rationally connected {C}alabi-{Y}au
  3-folds.
\newblock {\em Adv. Math.}, 378:107541, 32, 2021.

\bibitem[CHL22]{CHL22}
Guodu Chen, Jingjun Han, and Jihao Liu.
\newblock Uniform rational polytopes for {I}itaka dimensions.
\newblock {\em arXiv:2208.04663}, 2022.

\bibitem[DCS21]{DCS16}
Gabriele Di~Cerbo and Roberto Svaldi.
\newblock Birational boundedness of low-dimensional elliptic {C}alabi-{Y}au
  varieties with a section.
\newblock {\em Compos. Math.}, 157(8):1766--1806, 2021.

\bibitem[FG12]{FG12}
Osamu Fujino and Yoshinori Gongyo.
\newblock On canonical bundle formulas and subadjunctions.
\newblock {\em Michigan Math. J.}, 61(2):255--264, 2012.


\bibitem[Fil20]{Fil20}
Stefano Filipazzi.
\newblock On a generalized canonical bundle formula and generalized adjunction.
\newblock {\em Ann. Sc. Norm. Super. Pisa Cl. Sci.}, 21(5):1187--1221, 2020.

\bibitem[HJ22]{HJ22}
Jingjun Han and Chen Jiang.
\newblock Birational boundedness of rationally connected log {C}alabi–{Y}au
  pairs with fixed index.
\newblock {\em arXiv:2204.04946}, 2022.

\bibitem[HJL22]{HJL22}
Jingjun Han, Chen Jiang, and Yujie Luo.
\newblock Shokurov’s conjecture on conic bundles with canonical
  singularities.
\newblock {\em Forum of Mathematics, Sigma}, 10:e38, 2022.

\bibitem[HL22]{HL18}
Jingjun Han and Zhan Li.
\newblock Weak {Z}ariski decompositions and log terminal models for generalized
  polarized pairs.
\newblock {Math. Z.}, 302(2): 707--741, 2022.


\bibitem[HL20]{HanLiu20a}
Jingjun Han and Wenfei Liu.
\newblock On numerical nonvanishing for generalized log canonical pairs.
\newblock {\em Doc. Math.}, 25:93–123, 2020.

\bibitem[HL21]{HanLiu20}
Jingjun Han and Wenfei Liu.
\newblock On a generalized canonical bundle formula for generically finite
  morphisms.
\newblock {\em Annales de l'Institut Fourier}, 71(5):2047--2077, 2021.

\bibitem[HLL22]{HLL22}
Jingjun Han, Jihao Liu, and Yujie Luo.
\newblock {ACC} for minimal log discrepancies of terminal threefolds.
\newblock {\em arXiv:2202.05287}, 2022.

\bibitem[HLS19]{HLS19}
Jingjun Han, Jihao Liu, and Vyacheslav~V. Shokurov.
\newblock {ACC} for minimal log discrepancies of exceptional singularities.
\newblock {\em arXiv:1903.04338v1}, 2019.

\bibitem[HMX14]{HMX14}
Christopher~D. Hacon, James M\textsuperscript{c}Kernan, and Chenyang Xu.
\newblock A{CC} for log canonical thresholds.
\newblock {\em Ann. of Math. (2)}, 180(2):523--571, 2014.

\bibitem[HMX18]{HMX18}
Christopher~D. Hacon, James M\textsuperscript{c}Kernan, and Chenyang Xu.
\newblock Boundedness of moduli of varieties of general type.
\newblock {\em J. Eur. Math. Soc.}, 20(4):865--901, 2018.

\bibitem[HX15]{HX15}
Christopher~D. Hacon and Chenyang Xu.
\newblock Boundedness of log {C}alabi-{Y}au pairs of {F}ano type.
\newblock {\em Math. Res. Lett.}, 22(6):1699--1716, 2015.

\bibitem[Jia21]{Jiang21}
Chen Jiang.
\newblock A gap theorem for minimal log discrepancies of noncanonical
  singularities in dimension three.
\newblock {\em J. Algebraic Geom.}, 30:759--800, 2021.

\bibitem[JL21]{JL21}
Chen Jiang and Haidong Liu.
\newblock Boundedness of log pluricanonical representations of log
  {C}alabi-{Y}au pairs in dimension 2.
\newblock {\em Algebra Number Theory}, 15(2):545--567, 2021.

\bibitem[JLX22]{JLX22}
Junpeng Jiao, Jihao Liu, and Lingyao Xie.
\newblock On generalized lc pairs with b-log abundant nef part.
\newblock {\em arXiv:2202.11256}, 2022.

\bibitem[KM98]{KM98}
J\'anos Koll\'ar and Shigefumi Mori.
\newblock {\em Birational geometry of algebraic varieties}, volume 134 of {\em
  Cambridge Tracts in Mathematics}.
\newblock Cambridge University Press, Cambridge, 1998.
\newblock With the collaboration of C. H. Clemens and A. Corti, Translated from
  the 1998 Japanese original.

\bibitem[MP04]{MP04}
James M\textsuperscript{c}Kernan and Yuri Prokhorov.
\newblock Threefold thresholds.
\newblock {\em Manuscripta Math.}, 114(3):281--304, 2004.

\bibitem[Nak04]{Nak04}
Noboru Nakayama.
\newblock {\em Zariski-decomposition and abundance}, volume~14 of {\em MSJ
  Memoirs}.
\newblock Mathematical Society of Japan, Tokyo, 2004.

\bibitem[PS09]{PS09}
Yuri~G. Prokhorov and Vyacheslav~V. Shokurov.
\newblock Towards the second main theorem on complements.
\newblock {\em J. Algebraic Geom.}, 18(1):151--199, 2009.


\bibitem[Xu19]{XYN19}
Yanning Xu.
\newblock Some results about the index conjecture for log {C}alabi-{Y}au pairs.
\newblock {\em arXiv:1905.00297}, 2019.

\bibitem[Xu21]{Xu21}
Chenyang Xu.
\newblock K-stability of {F}ano varieties: an algebro-geometric approach.
\newblock {\em EMS Surv. Math. Sci.}, 8(1-2):265--354, 2021.

\end{thebibliography}
\end{document}